\documentclass[reqno, 10pt]{amsart}
\usepackage{todonotes}
\presetkeys{todonotes}{inline}{}
\makeatletter
\providecommand\@dotsep{5}
\makeatother
\usepackage{amssymb, amsmath, amsthm, enumerate,  marginnote, mathabx, mathrsfs, mathtools, tikz, bm, bbm}
\usetikzlibrary{arrows,calc}
\usetikzlibrary{arrows.meta, cd, patterns}
\usepackage{bbm} 
\usepackage{esint} 
\usepackage{verbatim}

\usepackage{hyperref}
\hypersetup{
    colorlinks=true,
    linkcolor=blue,
    filecolor=magenta,
}

\usepackage[normalem]{ulem}

\usetikzlibrary{decorations.pathreplacing}
\usetikzlibrary{decorations,calligraphy}

\usepackage{todonotes}

\usepackage{arydshln, multirow} 

\usepackage{comment}
\theoremstyle{plain}
\numberwithin{equation}{section}
\newtheorem{theorem}{Theorem}[section]
\newtheorem{corollary}[theorem]{Corollary}
\newtheorem{lemma}[theorem]{Lemma}
\newtheorem{proposition}[theorem]{Proposition}

\theoremstyle{definition}

\newtheorem*{induction hypothesis}{Induction hypothesis}

\theoremstyle{remark}
\newtheorem*{remark}{Remark}

\newcommand{\R}{\mathbb{R}}
\newcommand{\Z}{\mathbb{Z}}

\def\bbone{{\mathbbm 1}}

\newcommand{\Be}{\begin{equation}}
\newcommand{\Ee}{\end{equation}}

\newcommand{\Bm}{\begin{multline}}
\newcommand{\Em}{\end{multline}}
\newcommand{\Bea}{\begin{eqnarray}}
\newcommand{\Eea}{\end{eqnarray}}
\newcommand{\Beas}{\begin{eqnarray*}}
\newcommand{\Eeas}{\end{eqnarray*}}
\newcommand{\Benu}{\begin{enumerate}}
\newcommand{\Eenu}{\end{enumerate}}
\newcommand{\Bi}{\begin{itemize}}
\newcommand{\Ei}{\end{itemize}}

\def\intslash{\rlap{\kern  .32em $\mspace {.5mu}\backslash$ }\int}
\def\qsl{{\rlap{\kern  .32em $\mspace {.5mu}\backslash$ }\int_{Q_x}}}

\def\emph#1{{\it #1 }}

\begin{document}

\title[Young's convolution inequality on the hypercube]{
Optimal Young's convolution inequality and its reverse form on the hypercube}

\author{David Beltran}
\address{(D.B.) Departament d’An\`alisi Matem\`atica, Universitat de Val\`encia, Dr. Moliner 50, 46100 Burjassot, Spain}
\email{david.beltran@uv.es}

\author{Paata Ivanisvili}
\address{(P.I.) Department of Mathematics, University of California, 
Irvine, CA 92617, USA}
\email{pivanisv@uci.edu}

\author{Jos\'e Madrid}
\address[(J.M.)]{Department of  Mathematics, Virginia Polytechnic Institute and State University,  225 Stanger Street, Blacksburg, VA 24061-1026, USA
}
\email{josemadrid@vt.edu}

\author[Lekha Patil]{Lekha Patil}
\address{(L.P.) Department of Mathematics, University of California, 
Irvine, CA 92617, USA}
\email{patill@uci.edu}


\begin{abstract}
We establish sharp forms of Young's convolution inequality and its reverse on the discrete hypercube $\{0,1\}^d$ in the diagonal case $p = q$. As applications, we derive bounds for additive energies and sumsets. We also investigate the non-diagonal regime $p \neq q$, providing necessary conditions for the inequality to hold, along with partial results in the case $r = 2$.

\end{abstract}

\maketitle

\section{Introduction}

Let $d \geq 1$. In this paper we are interested in sharp versions of Young's convolution inequality (and its reverse version) in $\Z^d$ for functions supported on the hypercube $\{0,1\}^d$, which will yield applications in the theory of additive energies and sum-set bounds in $\{0,1\}^d$.

\subsection{Refined Young's convolution inequality}
Given $r \geq 1$, the classical Young's convolution inequality says that 
\begin{equation}\label{goal-inequality}
    \lVert f * g \rVert_{\ell^r(\Z^d)} \leq \lVert f \rVert_{\ell^p(\Z^d)} \, \lVert g \rVert_{\ell^q(\Z^d)}
\end{equation}
holds for all $f, \; g: \Z^d \to \R$ if $1 \leq p,q \leq \infty$ satisfy $1+\frac{1}{r}=\frac{1}{p}+\frac{1}{q}$. We are interested in improving the exponents $p$ and $q$ under the additional assumption that $f$ and $g$ are supported in $\{0,1\}^d$.

We start discussing the diagonal case $p=q$, for which we have a sharp result. 
\begin{theorem}\label{thm: f g r p p}
    Let $r\geq 1$ and $p_r:=\frac{2r}{\log_2(2+2^r)}$. Then
\begin{equation}\label{ineq: f*g r p p}
\|f*g\|_r\leq \|f\|_{p_r} \|g\|_{p_r}
\end{equation}
holds for all $f,g:\Z^d\to \R$ supported in $\{0,1\}^d$. Moreover, the exponent $p_r$ cannot be replaced by any larger real number.
\end{theorem}
The sharpness of $p_r$ can be seen from taking $f=g=\mathbbm{1}_{\{0,1\}^d}$. Note that the classical Young's convolution inequality provides \eqref{ineq: f*g r p p} with $p_r$ replaced by $\frac{2r}{r+1}$, which is a smaller number for $r>1$.

Theorem \ref{thm: f g r p p} immediately implies sharp off-diagonal bounds for $f=g$ via Hölder's inequality on the right-hand side of \eqref{ineq: f*g r p p}.

\begin{corollary}\label{cor:f f pqr}
Let $r\geq 1$ and $1 \leq p,q \leq \infty$ satisfying $\frac{1}{p} + \frac{1}{q}=\frac{\log_2(2+2^r)}{r}$. Then
\begin{equation}\label{ineq: f*f r p q}
\|f*f\|_r\leq \|f\|_{p} \|f\|_q
\end{equation}
holds for all $f,g:\Z^d\to \R$ supported in $\{0,1\}^d$.   
\end{corollary}

As discussed after Theorem \ref{thm: f g r p p}, the line $\frac{1}{p} + \frac{1}{q}=\frac{ \log_2(2+2^r)}{r}$ is sharp: taking $f=\mathbbm{1}_{\{0,1\}^d}$, one readily sees that \eqref{ineq: f*f r p q} fails if $\frac{1}{p} + \frac{1}{q} < \frac{\log_2(2^r+2)}{2}$. In the case of two different functions, it is not possible to obtain bounds for all  $1 \leq p,q \leq \infty$ in $\frac{1}{p} + \frac{1}{q}=\frac{ \log_2(2+2^r)}{r}$. Taking the functions $\mathbbm{1}_{\{0,1\}}$ and $\mathbbm{1}_{\{0\}}$ for $d=1$ immediately shows that one must have $p,q \leq r$. However, something stronger is true. 

\begin{proposition}\label{prop:necessary critical line}
Let $r > 1$ and $1\leq p,q \leq \infty$ be such that $\frac{1}{p}+\frac{1}{q}=\frac{\log_{2}(2^r+2)}{r}$. If \eqref{goal-inequality} holds for all $f,g:\Z^d\to \R$ supported in $\{0,1\}^d$, then $\frac{r+2^{r-1}}{1+2^{r-1}}\leq p,q\leq \frac{1}{\frac{\log_2 (2^r+2)}{r}-\frac{1+2^{r-1}}{r+2^{r-1}}}$.
\end{proposition}

We are able to prove the sharp bounds along the line $\frac{1}{p}+\frac{1}{q}=\frac{\log_{2}(2^r+2)}{r}$ for $r=2$.

\begin{theorem}\label{thm:r=2 critical line}
Let $1\leq p,q \leq \infty$ be such that $\frac{1}{p}+\frac{1}{q}=\frac{\log_{2}6}{2}$. Then
\begin{equation}\label{ineq: lem f*g 2 p q}
\|f*g\|_2\leq \|f\|_{p}\|g\|_q
\end{equation}
holds for all $f,g:\Z^d\to \R$ supported in $\{0,1\}^d$ if and only if $\frac{4}{3}\leq p,q\leq \frac{1}{\frac{\log_2 3}{2}-\frac{1}{4}}$.
\end{theorem}

One can interpolate the results obtained in the line $\frac{1}{p}+\frac{1}{q}=\frac{\log_{2}(2^r+2)}{r}$ with the bounds from the classical Young's convolution inequality for $p=1$ or $q=1$ to obtain further off-diagonal bounds. In the $r=2$, one obtains the following.

\begin{corollary}\label{cor:r=2}
Let $1\leq p,q \leq \infty$. Then
the inequality
\begin{equation}\label{ineq: thm f*g 2 p q}
\|f*g\|_2\leq \|f\|_{p}\|g\|_q
\end{equation}
holds for all $f,g:\Z^d\to \R$ supported in $\{0,1\}^d$ if $(1/p,1/q)\in \Omega$, where $\Omega$ is the pentagon with vertices $(1/2,1),(1,1),(1,1/2),(3/4,\frac{\log_2 3}{2}-\frac{1}{4})$ and $(\frac{\log_2 3}{2}-\frac{1}{4},3/4)$. Moreover, the inequality \eqref{ineq: thm f*g 2 p q} does not hold if $\frac{1}{p}+\frac{1}{q}<\frac{\log_2 6}{2}$ or $\min\{1/p,1/q\}<1/2$.
\end{corollary}

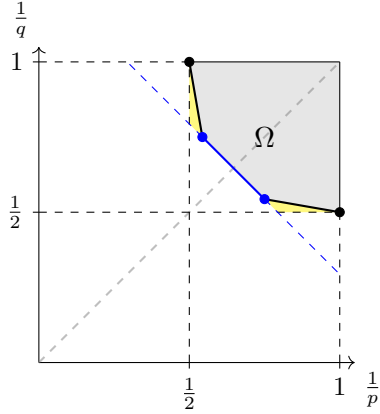
\begin{figure}
  \centering
  \begin{tikzpicture}[scale=4]
    \coordinate (A) at (0.75, 0.5435);
    \coordinate (B) at (0.5435, 0.75);
    \coordinate (D) at (1, 0.5);
    \coordinate (E) at (0.5, 1);
    \coordinate (TopRight) at (1,1);

    \coordinate (H) at (0.7935, 0.5); 
    \coordinate (K) at (0.5, 0.7935); 

    \fill[gray!20]
      (D) -- (TopRight) -- (E) -- (B) -- (A) -- cycle;

    \fill[yellow, opacity=0.5] (D) -- (A) -- (H) -- cycle; 

    \fill[yellow, opacity=0.5] (E) -- (B) -- (K) -- cycle; 

    \node at (0.75,0.75) {\large $\Omega$};

    \draw[->] (0,0) -- (1.05,0) node[below right] {$\frac{1}{p}$};
    \draw[->] (0,0) -- (0,1.05) node[above left] {$\frac{1}{q}$};

    \draw (1,0pt) -- (1,-0.02) node[below] {$1$};
    \draw (0pt,1) -- (-0.02,1) node[left] {$1$};

    \draw (0.5,0pt) -- (0.5,-0.02) node[below] {$\tfrac{1}{2}$};
    \draw (0pt,0.5) -- (-0.02,0.5) node[left] {$\tfrac{1}{2}$};

    \draw[thick, gray, dashed, opacity=0.5] (0,0) -- (1,1);

    \draw[dashed] (1,0) -- (1,0.5);
    \draw (1,0.5) -- (1,1);
    \draw[dashed] (0,1) -- (0.5,1);
    \draw (0.5,1) -- (1,1);

    \draw[dashed] (0,0.5) -- (1,0.5);
    \draw[dashed] (0.5,0) -- (0.5,1);

    \draw[blue, thick] (A) -- (B);
    \draw[blue, dashed] (B) -- (0.2935, 1);
    \draw[blue, dashed] (A) -- (1, 0.2935);
    \draw[black, thick] (A) -- (D);
    \draw[black, thick] (B) -- (E);

    \filldraw[blue] (A) circle (0.015);
    \filldraw[blue] (B) circle (0.015);

    \filldraw[black] (0.5,1) circle (0.015);
    \filldraw[black] (1,0.5) circle (0.015);

  \end{tikzpicture}
  \caption{The case $r=2$ in Corollary \ref{cor:r=2}. The \textcolor{blue}{blue} points represent the exponents $(1/p,1/q) = (3/4, \frac{\log_2 3}{2} - \frac{1}{4})$ and $(1/p,1/q) = (\frac{\log_2 3}{2} - \frac{1}{4}, 3/4)$. The convolution inequality $\|f \ast g\|_2 \leq \|f\|_p \|g\|_q$ for $f,g:\Z^d\to \R$ supported in $\{0,1\}^d$ holds for $(1/p,1/q)$ in the \textcolor{gray}{gray} region $\Omega$, and fails in the non-colored region. The validity of the inequality in the interior of the yellow triangles and the boundary segments $1/p=1/2$, $1/q=1/2$ remains open.}
  \label{fig:region-plot}
\end{figure}

Note that Corollary \ref{cor:r=2} does not completely characterise the exponents $1 \leq p,q \leq \infty$ for which the \eqref{ineq: thm f*g 2 p q} holds:  see Figure \ref{fig:region-plot}.

\subsection{Applications to $k$-higher additive energies}

The inequality \eqref{goal-inequality} can be re-interpreted as a bound for the $k$-higher additive energies of two sets when $f$ and $g$ are characteristic functions. Let $A, B \subseteq \{-1,1\}^d$ and $k \geq 2$ be an integer. Following \cite{Shkredov, SS}, we define the \textit{$k$-higher additive energy} of $A$ and $B$ as
\[
\widetilde{E}_{k}(A,B):=|\{(a_1,\dots,a_{k}, b_1, \dots, b_k) \in A^{k} \times B^k : a_1-b_1= \cdots = a_{k} - b_{k}\}|
\]
and observe that $\widetilde{E}_k(A,B)=\| \bbone_{A} \ast \widetilde{\bbone_B}\|_{k}^k$, where $\widetilde{f}(x):=f(-x)$ for any given $f : \{-1,1\}^{d} \to \R$. Bounds for $\widetilde{E}_k(A,B)$ when $A=B \subseteq \{-1,1\}^d$  were established by Kane and Tao \cite{KT} for $k=2$ and by De Dios, Ivanisvili, Greenfeld and Madrid \cite{DGIM} for a general integer $k \geq 2$. By a direct application\footnote{Theorem \ref{thm: f g r p p} equally applies to $\{-1,1\}^d$, which is more convenient for the application to additive energies because of the reflected function $\widetilde{1_B}$.} of Theorem \ref{thm: f g r p p} we recover their results and extend them to a pair of subsets $A,B$.

\begin{corollary}
    Let $d \geq 1$, $k \geq 2$, $q_k:=\log_2(2+2^k)$ and $A,B \subseteq \{-1,1\}^d$. Then $\widetilde{E}_k(A,B) \leq |A|^{q_k/2}|B|^{q_k/2}$. Furthermore, the exponent $q_k$ cannot be replaced by any smaller number.
\end{corollary}

In the same manner, we can use Proposition \ref{thm:r=2 critical line} to obtain a bound for the standard additive energy $(k=2)$ with different exponents on the cardinalities of $A$ and $B$.

\begin{corollary}\label{cor:additive k=2}
    Let $d \geq 1$ and $\frac{4}{3} \leq p, q \leq \frac{1}{\frac{\log_2 3}{2} - \frac{1}{4}}$ such that $\frac{1}{p}+\frac{1}{q}=\frac{\log_2 6}{2}$. Let $A,B \subseteq \{-1,1\}^d$. Then $\widetilde{E}_2(A,B) \leq |A|^{p}|B|^{q}$. Furthermore, the range for $p$ and $q$ is sharp.
\end{corollary}

We also refer to \cite{BCK} for other recent results related to additive energies on the hypercube.


\subsection{Sharp reverse Young's convolution inequality}

Young's convolution inequality \eqref{goal-inequality} does not hold in the non-Banach range $r<1$. However, Leindler \cite{Leindler} showed that its converse holds in this case. Namely, given $0<p,q, r < 1$ with $1+\frac{1}{r}=\frac{1}{p}+\frac{1}{q}$, the inequality
\begin{equation}\label{eq:reverse young general}
    \|f \ast g\|_{\ell^r(\Z^d)} \geq \|f\|_{\ell^p(\Z^d)} \|g\|_{\ell^q(\Z^d)}
\end{equation}
holds for all $f,g:\Z^d \to [0,\infty)$. As in the case $r>1$, the inequality \eqref{eq:reverse young general} can be sharpened for functions supported in $\{0,1\}^d$. The counterpart to Theorem \ref{thm: f g r p p} is the following. 

\begin{theorem}\label{thm: f g r p p reverse}
    Let $0 <r <  1$, and $p_r:=\frac{2r}{\log_2(2+2^r)}$. Then
\begin{equation}\label{rev-y01}
\|f*g\|_r\geq \|f\|_{p_r} \|g\|_{p_r}
\end{equation}
holds for all $f,g:\Z^d\to \R$ supported in $\{0,1\}^d$. Moreover, the exponent $p_r$ cannot be replaced by any smaller real number.
\end{theorem}

In the off-diagonal case, we also have an analogous necessary condition to that in Proposition \ref{prop:necessary critical line}.

\begin{proposition}\label{prop:necessary critical line r<1}
    Let $0<r < 1$ and $0 < p,q < 1$ be such that $\frac{1}{p}+\frac{1}{q}=\frac{\log_{2}(2^r+2)}{r}$. If
    \begin{equation*}
        \|f \ast g\|_{\ell^{r}(\Z^d)} \geq \| f \|_{\ell^p(\Z^d)} \|g\|_{\ell^q(\Z^d)}
    \end{equation*}
    holds for all $f,g:\Z^d\to [0,\infty)$ supported in $\{0,1\}^d$, then 
    $$\frac{1}{\frac{\log_2 (2^r+2)}{r}-\frac{1+2^{r-1}}{r+2^{r-1}}} \leq p,q\leq \frac{r+2^{r-1}}{1+2^{r-1}}. $$
\end{proposition}

The proofs of Theorem \ref{thm: f g r p p reverse} and Proposition \ref{prop:necessary critical line r<1} are closely related to those of Theorem \ref{thm: f g r p p} and Proposition \ref{prop:necessary critical line}, and will be treated simultaneously.

\subsection{Applications to sumsets}

Theorem \ref{thm: f g r p p reverse} has implications for lower bounding the size of the sumset for subsets of the discrete cube. Indeed, after applying the inequality (\ref{rev-y01}) to $f^{1/r}$ and $g^{1/r}$ and rising both sides to the power $r$ we see that 
\begin{align*}
    \| f^{1/r}*g^{1/r}\|_{r}^{r} \to \|f\overline{*}g\|_{1} \quad \text{as $r \to 0$},
\end{align*}
where 
\begin{align*}
    f\overline{*}g(x) =\sup_{u+v=x}f(u)g(v),
\end{align*}
and $\|f^{1/r}\|_{p_{r}}^{r} \|g^{1/r}\|_{p_{r}}^{r} \to \|f\|_{2/\log_{2}(3)} \|g\|_{2/\log_{2}(3)}$ as $r \to 0$. Thus, the {\em limiting} $r=0$ version of Theorem~\ref{thm: f g r p p reverse}  gives the following, which was obtained in \cite{BKIM}.
\begin{corollary}[Sharp Pr\'ekopa--Leindler on $\{0,1\}^{d}$]\label{cor011}
For all $f, g :\mathbb{Z}^{d} \to [0, \infty)$ supported on $\{0,1\}^{d}$ we have 
\begin{align*}
    \| f \overline{*} g\|_{1}\geq \|f\|_{2/\log_{2}(3)} \|g\|_{2/\log_{2}(3)}.
\end{align*}
\end{corollary}
Given $A, B \subseteq \{0,1\}^d$, we define the \textit{sumset} $A+B$ as
\begin{equation*}
    A + B := \{a + b : a \in A, b \in B\}.
\end{equation*}
Estimating the size of $A+B$ in terms of the size of $A$ and $B$ is a classical Brunn-Minkowski type question. A sharp result on the hypercube was obtained in \cite{Woodall} and \cite{HS} independently, which can be recovered by our Corollary \ref{cor011} through the elementary identity $|A+B|=\| 1_{A} \overline{*} 1_{B}\|_{1}$.

\begin{corollary}[Sharp Brunn--Minkowski on $\{0,1\}^{d}$]\label{sumset01}
    Let $d \geq 1$. Then
    \begin{equation}\label{eq:sumset}
        |A+B| \geq |A|^{\frac{\log_23}{2}}|B|^{\frac{\log_23}{2}}
    \end{equation}
    for all $A,B \subseteq \{0,1\}^d$.
\end{corollary}

The power $\frac{\log_{2}(3)}{2}$ in (\ref{eq:sumset}) is sharp and it can be see by choosing $A=B=\{0,1\}^{d}$.  We refer the reader to \cite{Bourgain} or the recent preprint \cite{BKIM} for versions of \eqref{eq:sumset} involving more than two sets.



\subsubsection*{Remark}  
While finalizing this manuscript, we learned that Crmari\v{c}, Kova\v{c}, and Shiraki \cite{CKS} have independently obtained Theorem~\ref{thm: f g r p p}. Their approach is distinct from ours and involves computer-assisted verification using Mathematica. In contrast, our proof is purely analytical and does not rely on computational tools. Furthermore, our method extends directly to establish the sharp reverse inequality stated in Theorem~\ref{thm: f g r p p reverse}, which plays a key role in deriving sumset bounds, as illustrated in Corollary~\ref{sumset01}.

\subsubsection*{Organisation of the paper} In Section \ref{sec:induction} we reduce convolution inequalities on $\{0,1\}^d$ to numerical inequalities. In Section \ref{sec:diagonal} we address the $p=q$ case and provide the proofs of Theorems \ref{thm: f g r p p} and \ref{thm: f g r p p reverse}. Finally, in Section \ref{sec:off-diagonal}, we address the $p \neq q$ case and provide the proof of the necessary condition in Propositions \ref{prop:necessary critical line} and \ref{prop:necessary critical line r<1} and of the bounds in Theorem \ref{thm:r=2 critical line} for $r=2$.

\subsubsection*{Acknowledgements}
The authors would like to thank Jaume de Dios for stimulating conversations. D.B. is supported by the grants RYC2020-029151-I and PID2022-140977NA-I00 funded by MICIU/AEI/10.13039/501100011033, ``ESF Investing in your future" and FEDER, UE. J.M. was partially supported by the AMS Stefan Bergman Fellowship and the Simons Foundation Grant $\# 453576$.  P.I. was supported in part by NSF CAREER grant DMS-2152401 and a Simons Fellowship.

\section{Reduction to numerical inequalities}\label{sec:induction}

We start by reducing the convolution inequalities to numerical inequalities via a standard induction argument. 

\begin{lemma}\label{lem:induction}
Let $d \geq 1$.
\begin{enumerate}[(i)]
    \item Let $r>1$, $p,q \leq r$ and assume that the inequality
    \begin{equation}\label{eq:numerical r>1}
    [ 1 + (x  + y)^r + (x y)^r  ]^{1/r} \leq (1+x^p)^{1/p} (1+y^q)^{1/q}
\end{equation}
holds for all $x,y \geq 0$. Then
\begin{equation}\label{eq:young general hypercube r>1}
    \|f \ast g\|_{\ell^r({\Z^d})} \leq \|f\|_{\ell^p(\Z^d)} \|g\|_{\ell^q(\Z^d)}
\end{equation}
holds for all $f,g:\Z^d \to \R$ supported in $\{0,1\}^d$.
\item  Let $0<r<1$, $p,q \geq r$ and assume that the inequality
    \begin{equation}\label{eq:numerical r<1}
    [ 1+ (x  + y)^r + (x y)^r  ]^{1/r} \geq (1+x^p)^{1/p} (1+y^q)^{1/q}
\end{equation}
holds for all $x,y \geq 0$. Then
\begin{equation}\label{eq:young general hypercube r<1}
    \|f \ast g\|_{\ell^r({\Z^d})} \geq \|f\|_{\ell^p(\Z^d)} \|g\|_{\ell^q(\Z^d)}
\end{equation}
holds for all $f,g:\Z^d \to [0,\infty)$ supported in $\{0,1\}^d$.
\end{enumerate}
\end{lemma}

\begin{proof}
We only show (i); the argument for (ii) is entirely analogous. By the triangle inequality, it suffices to prove \eqref{eq:young general hypercube r>1} for non-negative functions. Assume $r>1$ and that the inequality \eqref{eq:numerical r>1} holds. We first show that this implies the case $d=1$ of \eqref{eq:young general hypercube r>1}.  Setting $f_0=f(0)$, $f_1=f(1)$, $g_0=g(0)$ and $g_1=g(1)$, the case $d=1$  of \eqref{eq:young general hypercube r>1} amounts to showing
\begin{equation}\label{eq:reduction to numerical 1}
    [ (f_0 g_0)^r + (f_0g_1 + f_1 g_0)^r +(f_1 g_1)^r ]^{1/r} \leq (f_0^p+f_1^p)^{1/p} (g_0^q+g_1^q)^{1/q}
\end{equation}
for all $f_0,f_1,g_0,g_1 \geq 0$.  If either $f_1=0$ or $g_1=0$, the resulting inequality follows by the embedding $\ell^q \subseteq \ell^r$ for $q \leq r$ or $\ell^p \subseteq \ell^r$ for $p \leq r$, respectively. Thus, it suffices to prove \eqref{eq:reduction to numerical 1} for $f_1, g_1 > 0$. Setting $x=f_0/f_1$, $y=g_0/g_1$ and dividing by $f_1$ and $g_1$, \eqref{eq:reduction to numerical 1} becomes
\begin{equation*}
    [ (x y)^r + (x  + y)^r +1 ]^{1/r} \leq (x^p+1)^{1/p} (y^q+1)^{1/q}
\end{equation*}
for all $x,y \geq 0$, which is \eqref{eq:numerical r>1}. This shows \eqref{eq:young general hypercube r>1} for $d=1$.

The case $d >1$ follows by induction. Assume that \eqref{eq:young general hypercube r>1} holds on $\Z^{d-1}$, for functions supported on $\{0,1\}^{d-1}$. For $u,v \in \{0,1\}^{d}$, let $u = (u', u_d)$ where $u' \in \{0,1\}^{d-1}$, and given $f:\Z^d \to [0,\infty)$ supported in $\{0,1\}^d$ we denote $f_{u_d}(u'):=f(u)$, which is supported in $\{0,1\}^{d-1}$ (and similarly for $v$ and $g$). For $r>1$, we have by Minkowski's inequality, induction hypothesis, and the case $d=1$ that
\begin{align*}
&\|f \ast g \|_{\ell^{r}(\Z^d)}  = \Big( \sum_{m \in \{0,1,2\}^d} \Big| \sum_{\substack{u,v \in \{0,1\}^d \\u+v=m}} f(u)g(v) \Big|^r \Big)^{1/r} \\
&  = \Big( \sum_{m_d \in \{0,1,2\}} \sum_{m' \in \{0,1,2\}^{d-1}} \Big|\sum_{\substack{u_d,v_d \in \{0,1\} \\u_d+v_d=m_d}} \sum_{\substack{u',v' \in \{0,1\}^{d-1} \\u'+v'=m'}} f_{u_d}(u')g_{v_d}(v') \Big|^r \Big)^{1/r} \\
& \leq \Big( \sum_{m_d \in \{0,1,2\}} \Big[\sum_{\substack{u_d,v_d \in \{0,1\} \\u_d+v_d=m_d}}  \Big( \sum_{m' \in \{0,1,2\}^{d-1}} \Big|\sum_{\substack{u',v' \in \{0,1\}^{d-1} \\u'+v'=m'}} f_{u_d}(u')g_{v_d}(v') \Big|^r\Big)^{1/r} \Big]^{r}  \Big)^{1/r} \\
& \leq \Big( \sum_{m_d \in \{0,1,2\}} \Big[\sum_{\substack{u_d,v_d \in \{0,1\} \\u_d+v_d=m_d}}  \|f(\cdot, u_d)\|_{\ell^p(\Z^{d-1})} \|g(\cdot, v_d)\|_{\ell^q(\Z^{d-1})} \Big]^{r}  \Big)^{1/r} \\
& \leq \|f\|_{\ell^p(\Z^d)} \|g\|_{\ell^p(\Z^d)}
\end{align*}
which establishes \eqref{eq:young general hypercube r>1}.  For $r<1$, the argument is similar, using now the reversed inequalities as hypotheses and the fact that Minkowski's inequality also goes in the reverse direction.
\end{proof}

\section{The diagonal case}\label{sec:diagonal}

In this section we prove Theorems \ref{thm: f g r p p} and \ref{thm: f g r p p reverse} in conjunction. For a given $r >0$, we set $p_r:=\frac{2r}{\log_2(2+2^r)}$. We note that the case $r=1$ is immediate, since $p_1=1$ and the inequality actually becomes an equality. Thus, we may assume $r \neq 1$ throughout this section. Observe that if $r>1$, then $p_r < r$, whilst for $0<r<1$ we have $r < p_r$; this will play an important role in forthcoming arguments.

By Lemma \ref{lem:induction}, it suffices to verify the numerical inequalities \eqref{eq:numerical r>1} and \eqref{eq:numerical r<1} for $p=q=p_r$.  We first verify them when $x=y$. 

\begin{lemma}\label{lem:numerical f f p p}
    Let $r >0$, $r \neq 1$ and $p_r:=\frac{2r}{\log_2(2+2^r)}$. 
    Then the following hold.
    \begin{enumerate}[(i)]
        \item If $r > 1$, then
            \begin{equation}\label{eq:numerical f*f p p}
                (1+(2x)^{r}+x^{2r})^{1/r}\leq (1+x^{p_r})^{\frac{2}{p_r}}        
            \end{equation}
            for all $x \geq 0$.
        \item If $r < 1$, then
            \begin{equation}\label{eq:numerical f*f p p r < 1}
                (1+(2x)^{r}+x^{2r})^{1/r}\geq (1+x^{p_r})^{\frac{2}{p_r}}        
            \end{equation}
            for all $x \geq 0$.
    \end{enumerate}
\end{lemma}

\begin{proof}
It suffices to prove both \eqref{eq:numerical f*f p p} and \eqref{eq:numerical f*f p p r < 1} only for $x \in [0,1]$. If $x>1$, we can divide the inequalities by $x^2$ and perform the change of variables $1/x = y$. Then the resulting inequalities are the same but restricted to $[0,1]$.
    
Let $r>0, r \neq 1$. Raising \eqref{eq:numerical f*f p p} and \eqref{eq:numerical f*f p p r < 1} to the power $r$, taking logarithms, and setting $w=x^{r}$, proving the lemma is equivalent to determining that the function 
\begin{equation*}
    f(w):=\frac{2r}{p_r} \log (1+w^{p_r/r}) - \log(1+2^r w + w^2)
\end{equation*}
has constant sign (non-negative if $r>1$ and non-positive if $r <1$) on $[0,1]$. 
Clearly $f(0)=f(1)=0$ and
\begin{equation*}
    f'(w)= \frac{2 w^{p_r/r -1}}{1+w^{p_r/r}} - \frac{2^r + 2w}{1+2^r w + w^2} = \frac{2^r w^{p_r/r} + 2w^{p_r/r - 1} -2^r -2w}{(1+w^{p_r/r})(1+2^r w + w^2)}.
\end{equation*}
Note $f'(1)=0$. Since $f'(0)=+\infty$ if $r>1$ and $f'(0)=-2^r <0$ if $r<1$, it suffices to prove that $f'$ has at most one zero in $(0,1)$. The zeros of $f'$ in $(0,1)$ are equivalent to those of
\begin{equation*}
    g(w):=2^r w^{p_r/r} + 2w^{p_r/r - 1} -2^r -2w.
\end{equation*}

By direct computation,
\begin{align*}
    g'(w)&=\frac{2^{r}p_r}{r} w^{p_r/r -1} + 2\Big(\frac{p_r}{r}-1\Big)w^{p_r/r-2} - 2,
\\
    g''(w)&=2 \Big(\frac{p_r}{r}-1\Big)w^{p_r/r - 3} \Big[ \frac{2^{r-1}p_r}{r} w + \Big( \frac{p_r}{r} -2 \Big) \Big].
\end{align*}
Observe that $g''(w_r^*)=0$ if and only if $w_r^*=2^{-r+1}(2r/p_r-1)$. Clearly $w_r^*>0$. Moreover, $w_r^* < 1$ if and only if $r>1$. Note that this is equivalent to showing that
\begin{equation}\label{eq:w_r^*}
    \log_2(2+2^r) < 2^{r-1} +1
\end{equation}
if and only if $r>1$. Setting $b:=2^{r-1}$, this is equivalent to showing $h(b):=2^b-b-1>0$. Note that $h(1)=0$, and that $h'(b) \geq 0$ if and only if $b \geq \log_2(1/\log(2))$. Thus, $h(b)>0$ for $b >1$, so $w_r^* < 1$ if $r >1$. On the other hand, if $0<r<1$, we have $1/2<b<1$. Since $h(1/2)<0$, $h(1)=0$ and $h$ has only one critical point in $(1/2,1)$, we obtain $h(b)<0$ for $1/2<b<1$, and thus $w_r^*>1$ if $r<1$.

From the above considerations, we can readily deduce that $g$ has at most one zero on $(0,1)$ if $r<1$. Since $w_r^*>1$, the function $g''$ has constant sign on $(0,1)$. Noting that $g''(0)=-\infty$, we have that $g$ is concave. But $g(0)=-2^r$ and $g(1)=0$. This guarantees the claim and concludes the proof of (ii). 

We next consider the case $r>1$, for which $w_r^* \in (0,1)$. Indeed, note that $g''(w)>0$ if $w < w_{r}^*$ and $g''(w)<0$ if $w > w_r^*$. 
Combining this with the fact that $g'(0)=-\infty$ and $g'(1)>0$ (note that this is equivalent to \eqref{eq:w_r^*}), we have that $g'$ has exactly one zero on $(0,1)$. 
But $g(0)=+\infty$ and $g(1)=0$, and in view of the sign of $g'$ (with only one sign change from $-$ to $+$), we have that $g$ has 
exactly one zero in $(0,1)$, which concludes the proof of (i).
\end{proof}

We next show that it is possible to reduce the numerical inequalities \eqref{eq:numerical r>1} and \eqref{eq:numerical r<1} to the $x=y$ case.

\begin{proposition}\label{prop:reduction to diagonal}
    Let $r>0$, $r \neq 1$, and $p_r:=\frac{2r}{\log_2(2+2^r)}$. Consider the function
    \begin{equation}\label{eq:H def}
        H_r(x,y) :=1 + (x+y)^r + (xy)^r - [(1+x^{p_r})(1+y^{p_r})]^{\frac{r}{p_r}} 
    \end{equation}
    for all $x,y \geq 0$. The following hold.
    \begin{enumerate}[(i)]
        \item If $r > 1$, then $H_r(x,y) \leq H_r(\sqrt{xy},\sqrt{xy})$ for all $x,y \geq 0$.
        \item If $r < 1$, then $H_r(x,y) \geq H_r(\sqrt{xy},\sqrt{xy})$ for all $x,y \geq 0$.
    \end{enumerate}
\end{proposition}

\begin{proof}
For ease of notation, we set $H \equiv H_r$ and $p \equiv p_r$. We first claim that for all $x,y \geq 0$ with $x y >0$, we have
\begin{equation}\label{eq:H key}
    x \partial_x H(x,y) - y \partial_y H(x,y) \neq 0 \quad \text{if $x \neq y$}
\end{equation}
and show how to deduce the Proposition from this. Note that
\begin{align*}
&\frac{x\partial_x H(x,y)}{r} =  x(x+y)^{r-1}+y^rx^r - x^p(1+y^p)^\frac{r}{p}(1+x^p)^{\frac{r}{p}-1}\\
&\frac{y\partial_y H(x,y)}{r} = y(x+y)^{r-1} + x^ry^r - y^p(1+x^p)^\frac{r}{p}(1+y^p)^{\frac{r}{p}-1}.
\end{align*}
Thus 
\begin{equation}\label{eq:xH_x-yH_y}
    \frac{x \partial_x H - y\partial_y H}{r}=(x-y)(x+y)^{r-1} - (1+y^p)^{\frac{r}{p}}(1+x^p)^\frac{r}{p} \Big( \frac{x^p - y^p}{(1+x^p)(1+y^p)}\Big).
\end{equation}
Evaluating \eqref{eq:xH_x-yH_y} at $(x,y)=(0,1)$ we obtain
\begin{equation}\label{eq:xH_x-yH_y at (0,1)}
    x \partial_x H(x,y) - y\partial_yH(x,y)|_{(x,y)=(0,1)}=-r(1-2^{\frac{r}{p}-1}).
\end{equation}
Since $p < r$ if $r>1$, the expression \eqref{eq:xH_x-yH_y at (0,1)} is positive if $r>1$. On the other hand, since $p>r$ if $r < 1$, the expression \eqref{eq:xH_x-yH_y at (0,1)} is negative if $r <1$. Combining this with the claim \eqref{eq:H key} we obtain
\begin{equation}\label{eq:sign for xH_x-yH_y}
    \begin{cases}
        \text{ if $r > 1$, then $x\partial_x H(x,y)-y \partial_y H(x,y)>0$ for all $y>x$,} \\
        \text{ if $r < 1$, then $x \partial_x H(x,y)-y\partial_y H(x,y)<0$ for all $y>x$.}
    \end{cases}
\end{equation}
Now, for any given $0 <x < y $ we can find $A>0$ and $t>0$ such that $x=Ae^{-t}$ and $y=Ae^t$; note that $A=\sqrt{xy}$. Defining
\begin{equation*}
    R(s):=H(A e^{-s}, A e^s)
\end{equation*}
for $s  \geq 0$ we have
\begin{equation*}
    R'(s)=-Ae^{-s} H_x(A e^{-s},Ae^s) + Ae^s H_y (Ae^{-s}, Ae^s).
\end{equation*}
By \eqref{eq:sign for xH_x-yH_y}, for $s>0$ we have $R'(s) < 0$ if $r >1$ and $R'(s)>0$ if $r <1$. Consequently, for $s>0$ we have $R(s) \leq R(0)$ for $r > 1$ and $R(s) \geq R(0)$ for $r<1$. Since $A=\sqrt{xy}$, we conclude that (i) and (ii) in the statement of the proposition holds for all $0 < x < y $, and because the roles of $x$ and $y$ are symmetric and the case $x=y$ is immediate, it happens for all $x,y>0$. Since $H(0,y)=H(y,0)=1+y^r - (1+y^p)^{r/p}$ and $H(0,0)=0$, we conclude that (i) and (ii) holds for all $x,y \geq 0$, using that $p < r$ if $r >1$ and $p>r$ if $r <1$.

We next turn to the claim \eqref{eq:H key}, which will be proven by contradiction. Assume that $x \partial_x H(x,y)-y \partial_y H (x,y)=0$ for some $y>x>0$. In view of \eqref{eq:xH_x-yH_y}, this implies that 
\begin{align*}
     \frac{(x-y)(x+y)^{r-1}}{x^p - y^p} = (1+y^p)^{\frac{r}{p} - 1}(1+x^p)^{\frac{r}{p} - 1}
\end{align*}
for some $y > x >0$, which raised to the power $\frac{p}{r-p}$ yields 
\begin{equation}\label{eq:different 1}
    \left(\frac{(x-y)(x+y)^{r-1}}{x^p - y^p}\right)^\frac{p}{r-p} = (1+x^p)(1+y^p)
\end{equation}
for some $y> x>0$. Writing $y = tx$, the assumption \eqref{eq:different 1} can be rewritten as 
\begin{align*}
    x^p\left( \frac{ (t+1)^{r-1} (t-1)}{t^p - 1}\right)^\frac{p}{r-p} &= 1 + t^px^{2p} + x^p(1+t^p)
\end{align*}
for some $t>1$. 
Bringing everything to the right-hand side and multiplying by $x^{-2p}$, the above reads as
\begin{equation}\label{eq:different 2}
        x^{-2p} + x^{-p}\Big[1+t^p - \Big( \frac{ (t+1)^{r-1} (t-1)}{t^p - 1}\Big)^\frac{p}{r-p} \Big] + t^p =0
\end{equation}
for some $t>1$. Writing 
$z = x^{-p}$ and $Q(t): = \Big(\dfrac{ (t+1)^{r-1} (t-1)}{t^p - 1}\Big)^\frac{p}{r-p}$ in \eqref{eq:different 2}, we obtain the quadratic equation (in $z$)
\begin{equation*}
    f_t(z):=z^2 + z(1+t^p - Q(t)) + t^p = 0.
\end{equation*}
for some $t>1$. However, we will show that  $f_t(z)>0$ for all $z \geq 0$ and all $t>1$, which gives a contradiction and thus proves the claim \eqref{eq:H key}.

Note that the parabola $f_t(z)$ is positive at $z=0$ and $z=\infty$. Thus, to see that $f_t(z)>0$ for all $z>0$ and all $t>1$, it is enough to check positivity at the critical point  
\[z^*(t) := \frac{-(t^p + 1 - Q(t))}{2}.\]
Since we are only interested in $z>0$, we will assume
\begin{equation}\label{eq:sign of Q(t)}
t^p + 1 - Q(t) < 0,
\end{equation}
as otherwise the positivity of $f_t(z)$ for $z>0$ is immediate. Thus, we are left with proving 
\begin{equation}\label{eq:different 3}
    f_t(z^*(t))=\frac{-(t^p + 1 - Q(t))^2}{4} + t^p > 0
\end{equation}
for all $t>1$ satisfying \eqref{eq:sign of Q(t)}. Taking square roots, this is equivalent to showing that
\begin{equation*}
t^\frac{p}{2} - \frac{ |t^p + 1-Q(t)|}{2} > 0
\end{equation*}
and in view of \eqref{eq:sign of Q(t)} we have reduced \eqref{eq:different 3} to showing 
\begin{equation*}
    (1+t^\frac{p}{2})^2 > Q(t)=\left[ \frac{(t+1)^{r-1} (t-1)}{t^p - 1} \right]^{\frac{p}{r-p}}.
\end{equation*}
Taking $(r-p)/p$ roots, and noting that $r>p$ for $r>1$ and $r < p$ for $r<1$, it suffices to prove
\begin{equation}\label{eq:varphi goal}
    \begin{cases}
        \varphi(t) > 1  \text{ for all $t >1$ if $r >1$} \\
        \varphi(t) < 1 \text{ for all $t>1$ if  $r <1$}
    \end{cases}
\end{equation}
where 
\[\varphi(t) := \frac{(1+t^\frac{p}{2})^{\frac{2(r-p)}{p}} (t^p-1)}{(1+t)^{r-1}(t-1)}.  \]
Since $\varphi(t) = \varphi(1/t)$, it is enough to prove \eqref{eq:varphi goal} for all $t\in (0,1)$.

Next, observe that $p=\frac{2r}{\log_2(2+2^r)} < 2$ for all $r>0$. Thus, we have $t^{p/2} > t$ for $t \in (0, 1)$, which yields the bounds 
\begin{equation}\label{eq:varphi a>1}
    \varphi(t) > \frac{(t^p-1)(1+t)^{\frac{2(r-p)}{p}}}{(1+t)^{r-1}(t-1)}  \quad \text{ for $r >1$}
\end{equation}
and
\begin{equation}\label{eq:varphi a<1}
\varphi(t) < \frac{(t^p-1)(1+t)^{\frac{2(r-p)}{p}}}{(1+t)^{r-1}(t-1)} \quad \text{ for $r <1$}
\end{equation}
for all $0 < t < 1$. Define the function $g$ as the logarithm of the right-hand sides above, that is
\[g(t) := \log \left( \frac{(1-t^p)(1+t)^{\frac{2r}{p} - r - 1}}{1-t} \right) \] 
for $t \in [0, 1)$. By \eqref{eq:varphi a>1} and \eqref{eq:varphi a<1}, proving \eqref{eq:varphi goal} reduces to determining the sign of $g$. Since $g(0)=0$, it then suffices to show
\begin{equation}\label{eq:g' goal}
    \begin{cases}
        g'(t) \geq 0 \text{ for all } t \in (0,1) \text{ if $r >1$} \\
        g'(t) \leq 0 \text{ for all } t \in (0,1) \text{ if $r <1$}.
    \end{cases}
\end{equation}

We can rewrite
    \[g(t) = \log(1-t^p ) + \Big( \frac{2r}{p} - r - 1\Big) \log(1+t) - \log(1-t)\]
and therefore
    \begin{align*}
        g'(t) &= -\frac{p t^{p-1}}{1-t^p} + \Big( \frac{2r - rp - p}{p} \Big)  \frac{1}{1+t}  + \frac{1}{1-t} \\
        &= \frac{-p^2 t^{p-1} (1-t^2) + (2r-rp-p)(1-t^p)(1-t) + p(1-t^p)(1+t)}{p(1-t^p)(1-t^2)}.
    \end{align*}
    The denominator of $g'(t)$ is positive for all $t \in (0,1)$. Thus, the sign of $g'$ is determined by that of its numerator. We note the relation 
    \begin{align*}
        & -p^2 t^{p-1} (1-t^2) + (2r-rp-p)(1-t^p)(1-t) + p(1-t^p)(1+t) \\
        &= (r-p)(2-p) t^{p+1} - r(2-p)  t^p - p^2  t^{p-1} + [p(r+2) - 2r]  t + r(2-p)
    \end{align*}
    which allows to write the numerator of $g'(t)$ as $(2-p)t h(t)$, where
    \[h(t) := (r-p) t^{p} - r t^{p-1} - \Big( \frac{p^2}{2-p} \Big)  t^{p-2} + \Big( \frac{p(r+2) - 2r}{2-p} \Big)   + rt^{-1}\]
    is defined for $t \in (0,1)$. Since $p <2$, the sign of $g'$ is the same as that of $h$. Note that $h(1) = 0$, so to prove \eqref{eq:g' goal} it suffices to show that 
    \begin{equation}\label{eq:h goal}
    \begin{cases}
        h'(t) \leq 0 \text{ for all } t \in (0,1) \text{ if $r >1$} \\
        h'(t) \geq 0 \text{ for all } t \in (0,1) \text{ if $r <1$}.
    \end{cases}
\end{equation}
    We have 
    \begin{align*}
        h'(t) &= p(r-p)  t^{p-1} + r(1-p)  t^{p-2} + p^2  t^{p-3} - r t^{-2} \\
        &= t^{p-2}\Big( p(r-p)  t + r(1-p) + \frac{p^2}{t} - \frac{r}{t^p} \Big) = : t^{p-2}k(t).
    \end{align*}
Observe that $k(1) = 0$, so to prove \eqref{eq:h goal} it suffices to show
\begin{equation}\label{eq:k' goal}
    \begin{cases}
        k'(t) > 0 \text{ for all } t \in (0,1) \text{ if $r >1$} \\
        k'(t) < 0 \text{ for all } t \in (0,1) \text{ if $r <1$}.
    \end{cases}
\end{equation}
We note
\begin{align*}
    k'(t) &= p(r-p) - \frac{p^2}{t^2} + \frac{rp}{t^{p+1}} = p\left( r - p + \frac{r}{t^{p+1}} - \frac{p}{t^2} \right).
\end{align*}
Next, recall that $1 \leq p \leq 2$ for $r>1$. Thus, we have $t^p \leq t$ for  $t \in (0, 1)$ and $r>1$ and we obtain
\begin{align*}
    k'(t) 
    &\geq p \Big( r - p + \frac{r}{t^2} - \frac{p}{t^2} \Big) = p(r-p)\Big(1 + \frac{1}{t^2} \Big) > 0
\end{align*}
using that  $p < r$ for $r>1$. This verifies the first part of \eqref{eq:k' goal}. For $r<1$, we have $p<1$ and therefore $t^2 \leq t^{p+1}$ for $t \in (0,1)$. Thus
\begin{align*}
    k'(t)  \leq p\left(r-p+\frac{r}{t^{p+1}}-\frac{p}{t^{p+1}} \right)=p(r-p)\Big(1+\frac{1}{t^{p+1}}\Big)<0
\end{align*}
using that $r < p$ for $r <1$. This establishes the second part of \eqref{eq:k' goal}, and therefore concludes the proof of the claim \eqref{eq:H key} and of the proposition. 
\end{proof}

Lemma \ref{lem:numerical f f p p} and Proposition \ref{prop:reduction to diagonal} imply the numerical inequalities \eqref{eq:numerical r>1} and \eqref{eq:numerical r<1} for $p=q=p_r$. By Lemma \ref{lem:induction}, these imply the desired Theorems \ref{thm: f g r p p} and \ref{thm: f g r p p reverse}.

\begin{proof}[Proof of Theorem \ref{thm: f g r p p} and Theorem \ref{thm: f g r p p reverse}]
    By Lemma \ref{lem:induction} it suffices to prove the inequalities \eqref{eq:numerical r>1} and \eqref{eq:numerical r<1} for $p=q=p_r$. Note that these inequalities amount to saying that $H_r(x,y) \leq 0$ and $H_r(x,y) \geq 0$ respectively, where $H$ is as in \eqref{eq:H def}. Combining Proposition \ref{prop:reduction to diagonal} (i), and Lemma \ref{lem:numerical f f p p} (i) (with $x$ replaced by $\sqrt{xy}$) we have 
    \begin{equation*}
        H_r(x,y) \leq H_r(\sqrt{xy}, \sqrt{xy}) \leq 0
    \end{equation*}
    for $r >1$, as desired. 
    The case $r<1$ is analogous but applying parts (ii) of Proposition \ref{prop:reduction to diagonal} and Lemma \ref{lem:numerical f f p p}.

    We next turn to the sharpness of $p_r$. Assume $r>1$. As seen in the proof of Lemma \ref{lem:induction}, the inequality \eqref{goal-inequality} for $d=1$ and $p=q$ is equivalent to showing
    \begin{equation*}
    [1+(x+y)^r+(xy)^r]^{1/r}\leq (1+x^p)^{1/p}(1+y^p)^{1/p}.
    \end{equation*}
    Taking $x=y=1$, one readily obtains $\frac{2}{p} \geq \frac{\log_2(2+2^r)}{r}$, that is $p \leq p_r$. The case $r<1$ is entirely analogous.
\end{proof}

\section{The off-diagonal case}\label{sec:off-diagonal}

In this section we address Young's convolution inequalities when $p \neq q$, and provide the proofs for Propositions \ref{prop:necessary critical line} and \ref{prop:necessary critical line r<1} and Theorem \ref{thm:r=2 critical line}.

\subsection{Necessary conditions}\label{subsec:nec}

We provide a non-trivial necessary condition along the line $\frac{1}{p}+\frac{1}{q}=\frac{r}{\log_2(2+2^r)}$.

\begin{proof}[Proof of Proposition \ref{prop:necessary critical line} and Proposition \ref{prop:necessary critical line r<1}]
We first provide the proof of Proposition \ref{prop:necessary critical line}. Assume $r>1$. As seen in the proof of Lemma \ref{lem:induction}, the inequality \eqref{goal-inequality} for $d=1$ is equivalent to proving that
\begin{align*}
    [1+(x+y)^r+(xy)^r]^{1/r}\leq (1+x^p)^{1/p}(1+y^q)^{1/q}
\end{align*}
for all $x,y \geq 0$. In particular, if $y=1$, this reduces to
\begin{equation}\label{ineq: y=1}
1+(x+1)^r+ x^r\leq 2^{r/q}(1+x^p)^{r/p}
\end{equation}
for all $x \geq 0$. Note that taking $x=1$ we get $\frac{1}{p}+\frac{1}{q}\geq\frac{\log_{2}(2^r+2)}{r}$. 
Assuming that $\frac{1}{p}+\frac{1}{q}=\frac{\log_{2}(2^r+2)}{r}$, we have that \eqref{ineq: y=1} is equivalent to
\begin{equation*}
 \frac{1+(x+1)^r + x^r}{2^r+2}\leq \Big(\frac{1}{2}+\frac{x^p}{2}\Big)^{r/p}   
\end{equation*}
for all $x\geq 0$. Taking logarithms, this is equivalent to proving that
\begin{equation}\label{ineq: proof of main thm 2 funciones p q con log}
g(x):=\frac{r}{p}\log\Big(\frac{1}{2}+\frac{x^p}{2}\Big)-\log[x^r+(x+1)^r+1]+\log(2^r+2)\geq 0
\end{equation}
for all $x\geq0$. Observe that $g(1)=0$ and
\begin{align*}
g'(x) & =\frac{rx^{p-1}}{1+x^p}-\frac{rx^{r-1}+r(x+1)^{r-1}}{x^r+(x+1)^r+1} \\
& =r\frac{(x+1)^{r-1}x^{p-1}+x^{p-1}-(x+1)^{r-1}-x^{r-1}}{(1+x^p)(x^r+(x+1)^r+1)}.    
\end{align*}
In order for \eqref{ineq: proof of main thm 2 funciones p q con log} to hold, that is, $g(x) \geq 0=g(1)$, we require that $g'(1-\varepsilon) \leq 0$ and $g'(1+\varepsilon) \geq 0$ for all $\varepsilon>0$ sufficiently small. Note that $g'(1)=0$; thus we require $g''(1) \geq 0$. 

Let $h$ denote the numerator of $g'/r$,  that is $$h(x):=(x+1)^{r-1}x^{p-1}+x^{p-1}-(x+1)^{r-1}-x^{r-1}$$ for all $x\geq0$. We have $h(1)=0$ and that the sign of $g''(1)$ is the same as that of $h'(1)$. By direct computation
\begin{align*}
    h'(x)&=(r-1)(x+1)^{r-2}x^{p-1}+(p-1)(x+1)^{r-1}x^{p-2}\\
    &\ \ \ \ \ +(p-1)x^{p-2}-(r-1)(x+1)^{r-2}-(r-1)x^{r-2}.
\end{align*}
Thus $h'(1)=(p-1)(2^{r-1}+1)-(r-1)$ and we have $h'(1)\geq 0$ if and only if $p\geq \frac{r+2^{r-1}}{1+2^{r-1}}$. Then
$$
\frac{\log_2 (2^r+2)}{r}=\frac{1}{q}+\frac{1}{p}\leq \frac{1+2^{r-1}}{r+2^{r-1}}+\frac{1}{q},
$$
which implies that $q\leq \frac{1}{\frac{\log_2 (2^r+2)}{r}-\frac{1+2^{r-1}}{r+2^{r-1}}}$. Interchanging the roles of $p$ and $q$ in the previous calculations, we obtain that
$\frac{r+2^{r-1}}{1+2^{r-1}}\leq p,q\leq \frac{1}{\frac{\log_2 (2^r+2)}{r}-\frac{1+2^{r-1}}{r+2^{r-1}}}$ are necessary conditions for \eqref{goal-inequality} to hold when $\frac{1}{p}+\frac{1}{q}=\frac{r}{\log_2(2+2^r)}$ and $r>1$. This concludes the proof of Proposition \ref{prop:necessary critical line}.

The proof of Proposition \ref{prop:necessary critical line r<1} is entirely analogous. The only difference is that now one requires $h'(1)\leq 0$.
\end{proof}

\begin{remark}
The proof of Theorem \ref{thm: f g r p p} consisted in a reduction to the case $f=g$. The above necessary conditions reveal that Proposition \ref{prop:reduction to diagonal} cannot hold in the $p \neq q$ case unless one assumes $p,q$ as in the statement of Proposition \ref{prop:necessary critical line}, as otherwise combined with Corollary \ref{cor:f f pqr} would imply false estimates.
\end{remark}

\subsection{The case $r=2$}

We next show that the necessary conditions on the line $\frac{1}{p} + \frac{1}{q}=\frac{\log_2(2^r +2)}{r}$ are also sufficient for $r=2$.

\begin{proof}[Proof of Theorem \ref{thm:r=2 critical line}]
It suffices to prove \eqref{ineq: lem f*g 2 p q} for $p=\frac{4}{3}$ and $q=\frac{1}{\frac{\log_2 3}{2}-\frac{1}{4}}$, which by Lemma \ref{lem:induction} reduces to
\begin{equation}\label{eq:r=2 goal}
[1+(x+y)^2+(xy)^2]^{2}\leq (1+x^{4/3})^{3}(1+y^q)^{4/q}    
\end{equation}
for all $x,y\geq 0$. Once this is established, the remaining bounds follow by interpolation of this with the equivalent inequality for $q=\frac{4}{3}$ and $p=\frac{1}{\frac{\log_2 3}{2}-\frac{1}{4}}$ (which follows from interchanging the roles of $f$ and $g$). 

We turn to the proof of \eqref{eq:r=2 goal}. We assume $y>0$, as otherwise the inequality trivially follows since $q < 2$. For a fixed $y>0$, we define $f_{y}(x):=\frac{[1+(x+y)^2+(xy)^2]^{2}}{(1+x^{4/3})^{3}}$. We observe that $f_y(0)=f_y(+\infty)=(1+y^2)^2\leq (1+y^{q})^{4/q}$ since $q \leq 2$. 
A computation shows that
\begin{equation*}
    f_y'(x)=4 \left(\frac{1+(x+y)^2+(xy)^2}{(1+x^{4/3})^4}\right)g_y(x)    
\end{equation*}
where
\begin{align*}
    g_y(x):&=(xy^2+x+y)(1+x^{4/3})-x^{1/3}\big(1+(x+y)^2+(xy)^2\big) \\
    & = xy^2+x+y-x^{4/3}y-x^{1/3}y^2-x^{1/3} \\
    & = (x^{2/3}-1)[(y^2+1)x^{1/3}-y(1+x^{2/3})] \\
    & = (x^{2/3}-1)(-y)(x^{1/3}-y)(x^{1/3}-1/y).
\end{align*}
Thus $f_y'(x)=0$ if and only if $x\in\{1,y^3,\frac{1}{y^3}\}$. We observe that $x=1$ corresponds to a local minimum of $f_y$ for $y \neq 1$. Note that the sign of $f_y''(1)$ is the same as that of $g_y'(1)$. Since
\begin{equation*}
    g_y'(x)=y^2+1-\frac{1}{3}x^{-2/3}y^2 - \frac{4}{3}x^{1/3}y - \frac{1}{3}x^{-2/3}
\end{equation*}
we have $g_y'(1)=\frac{2}{3}(y-1)^2$, and so $f_y''(1)>0$ if $y \neq 1$, which means $x=1$ is a local minimum for $y \neq 1$. Since the inequality \eqref{eq:r=2 goal} holds trivially for $x=y=1$, it suffices to show that $\max \{f_y(y^3), f_y(1/y^3)\} \leq (1+y^q)^{4/q}$ holds for all $y > 0$. But
\begin{equation}\label{eq:1/y^3 reduction}
f_y(1/y^3)=\frac{\big(1+(\frac{1}{y^3}+y)^2+\frac{1}{y^4}\big)^{2}}{(1+\frac{1}{y^4})^{3}}=\frac{(y^8+y^6+2y^4+y^2+1)^{2}}{(y^{4}+1)^{3}}=f_y(y^3) 
\end{equation}
for all $y>0$; consequently, we have reduced \eqref{eq:r=2 goal} to proving $f_y(y^3) \leq (1+y^q)^{4/q}$ for all $y>0$. Indeed, it suffices to prove it only for $0 < y \leq 1$. If $y>1$ we can change variables $y \mapsto 1/y$ and note that, similarly to \eqref{eq:1/y^3 reduction}, we have
\begin{equation*}
    y^4 f_{1/y}(1/y^3)=f_y(y^3).
\end{equation*}

Thus, it suffices to show
$$
f_y(y^3)=\frac{(y^8+y^6+2y^4+y^2+1)^{2}}{(1+y^4)^{3}}\leq (1+y^{q})^{4/q}
$$
for all $y \in (0,1]$ or, equivalently,
$$
(1+y^2+y^4)^2\leq (1+y^4)(1+y^{q})^{4/q}
$$
for all $y\in(0,1]$. Define $T(y):=\frac{4}{q}\log(1+y^q)+\log(1+y^4)-2\log(1+y^2+y^4)$
for all $y\geq0$. Observe that $T(0)=T(1)=0$. A computation shows that
\begin{equation*}
T'(y)=4\frac{y^{q-1}+y^{q+3}+y^{q+5}-y-y^3-y^7}{(1+y^q)(1+y^4)(1+y^2+y^4)}   
\end{equation*}
and therefore $T'(\varepsilon)>0$ for all $\varepsilon>0$ sufficiently small and $T'(1)=0$. To show that $T(y)\geq 0$ in $[0,1]$, it suffices to prove that $T'$ has exactly one zero in $(0,1)$. This is equivalent to proving that
$$
R(y):=\log(1+y^4+y^6)-\log[y^{2-q}(1+y^2+y^6)]
$$
has exactly one zero in $(0,1)$. We have that $R'(y)$ equals to
$$
\frac{-(2-q)y^{12}-(4-q)y^{10}+(2+q)y^8+(3q-4)y^6+(2+q)y^4-(4-q)y^2-(2-q)}{y(1+y^4+y^6)(1+y^2+y^6)}.
$$
Then, by Descartes' rule of sign change and the value of $q=\frac{1}{\frac{\log_2 3}{2} - \frac{1}{4}}$, the numerator of $R'$ has two positive zeros. Since $R'(0) = -\infty$, $R'(1)=q-4/3>0$ and  $R'(y)<0$ when $y \to +\infty$, we have that $R'$ has one root in $(0,1)$ and another root in $(1,+\infty)$. This means that $R'$ has only one sign change, from $-$ to $+$, in $(0,1)$ and since $R(0)=+\infty$ and $R(1)=0$, we conclude that $R$ has only one zero in $(0,1)$.
\end{proof}

The proof of Corollary \ref{cor:r=2} is immediate by interpolation so we omit the details. The necessary conditions have already been discussed in Section \ref{subsec:nec}.




\bibliography{Reference}

\begin{thebibliography}{10}

\bibitem{BKIM}
Lars Becker, Paata Ivanisvili, Dmitry Krachun, and Jóse Madrid.
\newblock Discrete {B}runn-{M}inkowski inequality for subsets of the cube.
\newblock Preprint: \verb+arXiv:2404.04486+.

\bibitem{BCK}
Adrian Beker, Tonći Crmarić, and Vjekoslav Kovač.
\newblock Sharp estimates for {G}owers norms on discrete cubes.
\newblock Preprint: \verb+arXiv:2409.12579+.

\bibitem{Bourgain}
Jean Bourgain, Stephen Dilworth, Kevin Ford, Sergei Konyagin, and Denka
  Kutzarova.
\newblock Explicit constructions of {RIP} matrices and related problems.
\newblock {\em Duke Math. J.}, 159(1):145--185, 2011.

\bibitem{CKS}
Ton\'ci Crmari\v{c}, Vjekoslav Kova\v{c}, and Shobu Shiraki.
\newblock Inequalities in {F}ourier analysis on binary cubes.
\newblock Preprint: \verb+arXiv:2507.01359+.

\bibitem{DGIM}
Jaume de~Dios~Pont, Rachel Greenfeld, Paata Ivanisvili, and Jos\'e Madrid.
\newblock Additive energies on discrete cubes.
\newblock {\em Discrete Anal.}, pages Paper No. 13, 16, 2023.

\bibitem{HS}
D.~Hajela and P.~Seymour.
\newblock Counting points in hypercubes and convolution measure algebras.
\newblock {\em Combinatorica}, 5(3):205--214, 1985.

\bibitem{KT}
Daniel Kane and Terence Tao.
\newblock A bound on partitioning clusters.
\newblock {\em Electron. J. Combin.}, 24(2):Paper No. 2.31, 13, 2017.

\bibitem{Leindler}
L.~Leindler.
\newblock On a certain converse of {H}\"older's inequality.
\newblock In {\em Linear operators and approximation ({P}roc. {C}onf., {M}ath.
  {R}es. {I}nst., {O}berwolfach, 1971)}, volume Vol. 20 of {\em Internat. Ser.
  Numer. Math.}, pages 182--184. Birkh\"auser Verlag, Basel-Stuttgart, 1972.

\bibitem{SS}
Tomasz Schoen and Ilya~D. Shkredov.
\newblock Higher moments of convolutions.
\newblock {\em J. Number Theory}, 133(5):1693--1737, 2013.

\bibitem{Shkredov}
Ilya Shkredov.
\newblock Energies and structure of additive sets.
\newblock {\em Electron. J. Combin.}, 21(3):Paper 3.44, 53, 2014.

\bibitem{Woodall}
D.~R. Woodall.
\newblock A theorem on cubes.
\newblock {\em Mathematika}, 24(1):60--62, 1977.

\end{thebibliography}

\bibliographystyle{plain}
\end{document}